\documentclass[11pt,reqno,letterpaper]{amsart}

\usepackage{amssymb,amsmath,amsthm,amsfonts}
\usepackage{bbm,enumerate}
\usepackage{mathtools} 
\usepackage{xcolor}

\usepackage{bookmark}
\usepackage{hyperref}
\hypersetup{pdfstartview={FitH}}

\theoremstyle{plain}
\newtheorem{theorem}{Theorem}
\newtheorem{definition}[theorem]{Definition}
\newtheorem{question}{Question}
\newtheorem{lemma}[theorem]{Lemma}
\newtheorem{corollary}[theorem]{Corollary}

\theoremstyle{remark}
\newtheorem{remark}{Remark}

\numberwithin{equation}{section}

\newcommand{\eps}{\varepsilon}
\newcommand{\N}{\mathbb{N}}
\newcommand{\Z}{\mathbb{Z}}
\newcommand{\R}{\mathbb{R}}
\newcommand{\Q}{\mathbb{Q}}

\newcommand{\diam}{\operatorname{diam}}

\DeclarePairedDelimiter\abs{\lvert}{\rvert}

\begin{document}

\title[Translation invariant measure]{Impossibility of decoding a translation invariant measure from a single set of positive Lebesgue measure}

\author[A. Bulj]{Aleksandar Bulj}

\address{Department of Mathematics, Faculty of Science, University of Zagreb, Bijeni\v{c}ka cesta 30, 10000 Zagreb, Croatia}
\email{aleksandar.bulj@math.hr}

\subjclass[2020]{Primary 28A12; 
} 

\keywords{Translation invariant measure}

\begin{abstract}
    Let $\mu$ be a translation invariant measure on $(\mathbb{R}^d,\mathcal{B}(\mathbb{R}^d))$ and let $\lambda$ denote the Lebesgue measure on $\mathbb{R}^d$. If there exists an open set $U$ such that $0<\mu(U)=\lambda(U)<\infty$, it is a simple exercise to show that $\mu=\lambda|_{\mathcal{B}(\mathbb{R}^d)}$. Is the same conclusion true if $U$ is merely a Borel set?

    The main purpose of this short note is to construct a measure that provides a negative answer to this question. 
    Incidentally, this construction provides a new example of a translation invariant measure with a rich domain and range that is not Hausdorff, a problem previously studied by Hirst.

\end{abstract}

\maketitle


\section{Introduction}
Let $\mathcal{B}(\R^d)$ be the Borel $\sigma$-algebra on $\R^d$ and let $\lambda$ be the Lebesgue measure on $\R^d$. Let $\mu$ be another measure defined on $\mathcal{B}(\R^d)$ that is translation invariant. 
It is a simple exercise to check that if there exists a set $U\subset \R^d$ that is either a box (Cartesian product of $d$ intervals) or an open set, such that $0<\mu(U)=\lambda(U)<\infty$, then $\mu=\lambda$ (see Lemma \ref{lem:simple_equality}).
A natural generalization is the following question.

\begin{question}
\label{qu:main}
    Let $\mu$ be a translation invariant measure on $(\R^d,\mathcal{B}(\R^d))$. If there exists a set $U\in \mathcal{B}(\R^d)$ such that $0<\mu(U)=\lambda(U)<\infty$, does it imply that $\mu=\lambda$?
\end{question} 

The main purpose of this note is to present a construction of a measure that provides a negative answer to the given question, but which might be of independent interest, especially in light of the book \cite{SchillingBook}. 

\begin{theorem}
\label{thm:main}
    There exists a translation invariant measure $\mu$ on $(\R^d,\mathcal{B}(\R^d))$ that is not equal to the Lebesgue measure $\lambda$, but for which there exists a compact set $K\in \mathcal{B}(\R^d)$ such that $0<\mu(K)=\lambda(K)<\infty$. Specifically, this measure satisfies $\mu([0,1]^d)=+\infty$.
\end{theorem}

The construction follows the standard construction of a measure from an outer measure presented, for example, in \cite{FollandBook}, with the only novelty being the intricate definition of the domain of the pre-measure; see Remark \ref{rem} for more details.

Related to the question of our study, Hirst \cite{Hirst66} studied translation invariant measures that are not Hausdorff. 
He mentions a few simple examples of measures that are translation invariant and not Hausdorff: ``Translation invariant measures which are not Hausdorff measures can be specified
quite simply; we can have measures taking only the values zero and infinity, or
measures taking only a finite number of values, but having very few measurable sets." The main result of \cite{Hirst66} is a construction of a regular translation invariant metric outer measure that is not Hausdorff.
Incidentally, the measure in Theorem \ref{thm:main} is another example of a translation invariant measure that has rich domain, i.e. $\mathcal{B}(\R^d)$ and rich range, i.e. $[0,\infty]$. More precisely, the following corollary will be proved as a simple consequence of Theorem \ref{thm:main}.

\begin{corollary}
\label{cor:Hausd}
    The measure that satisfies Theorem \ref{thm:main} is a translation invariant measure on $(\R^d,\mathcal{B}(\R^d))$ that is not a Hausdorff measure and whose range is equal to $[0,\infty]$.
\end{corollary}


\subsection{Preliminaries}
A set $I\subset \R^d$ is called a \emph{box} if there exist intervals $(I_j)_{j=1}^d$ in $\R$  such that  $I=I_1\times \dots \times I_d$ (intervals may be open or closed, and they can be bounded or unbounded).
We say that a collection of sets is a \emph{ring} if it is closed under set differences and finite unions. We say that a ring is \emph{generated by} a collection of sets $\mathcal{C}$ if it is the smallest ring containing $\mathcal{C}$.
For a set $A\subset \R$ we denote $\diam(A):=\sup_{x,y\in A}\abs{x-y}$ with the convention $\diam(\emptyset)=0$.
Given a right-continuous increasing function $h:[0,\infty) \to [0,\infty]$ with $h(0)=0$, define
\[\nu^{*}(E):=\sup_{\delta >0} \nu_{\delta}^{*}(E),\]
where
\[\nu_\delta^{*}(E):=\inf\left\{\sum_{j=1}^{\infty} h(\diam(E_j)): E\subset \bigcup_{j\in \N} E_j,\; \diam(E_j)< \delta\right\}.\]
We say that the measure $\nu$ is the \emph{Hausdorff measure corresponding to $h$} if it is restriction of the outer measure $\nu^{*}(E)$ to measurable sets. For a detailed discussion of the properties of Hausdorff measures, we refer the reader to \cite{RogersBook}.

\section{Proofs}
The following lemma was mentioned in the introduction and will be used later. Since its proof relies on standard arguments from measure theory, we provide only a brief outline.
\begin{lemma}
\label{lem:simple_equality}
    Let $\mu$ be a translation invariant measure on $(\R^d,\mathcal{B}(\R^d))$. If there exists a set $U\subset \R^d$ that is either a box or an open set, such that $0<\mu(U)=\lambda(U)<\infty$, then $\mu = \lambda\vert_{\mathcal{B}(\R^d)}$.
\end{lemma}

\begin{proof}[Proof]
    Assume that $U\subset\R^d$ is an open set - the case when $U$ is a box will be proved along the same lines.
    Since every open set can be expressed as a countable union of dyadic cubes, there exists a cube $I\subset U$ such that $0<\mu(I)<\infty$. 
    Next, note that all of faces of $I$ must have $\mu$-measure equal to $0$. 
    Indeed, if that was not true, one could pick a countable disjoint union of translates of a given face that are subsets of $I$ to conclude that the measure of the box $I$ would be infinite.

    Denote $c:= \mu(I)/\lambda(I)$ and assume, because of the previous part and translation invariance, that $I=[0,a_1)\times\dots\times [0,a_d)$. 
    For any $d$-tuple $q=(q_1,\dots, q_d)\in \Q_+^d$, by double counting measure of boxes of equal size, one can easily verify that 
    $\mu(I_q) = c \lambda(I_q)$, where $I_q:=[0,q_1a_1)\times \dots \times [0,q_da_d)$.
    By monotonicity of measure one can extend this equality to $d$-tuples of positive real numbers, what is, using translation invariance, equivalent to $\mu(J)=c\lambda(J)$ for any box $J \subset \R^d$. Since $\lambda$ is $\sigma$-finite, and boxes generate $\mathcal{B}(\R^d)$, one can conclude that $\mu = c\lambda$ on $\mathcal{B}(\R)$ as in \cite[Theorem~1.14]{FollandBook}. Finally, testing the equality on $U$, one concludes that $c=1$, so $\mu = \lambda\vert_{\mathcal{B}(\R^d)}$.
\end{proof}

The following ring of sets is the key ingredient 
in the construction of the counterexample.
\begin{definition}
\label{def:ring}
    We define the ring $\mathcal{R}$ in the following way.
    Let $C$ be any fixed compact set that is nowhere dense in $\R^d$ and has positive Lebesgue measure (it is well known that such sets exist - they are called fat Cantor sets). Define $\mathcal{R}$ to be the ring generated by the following collection of sets:
    \[\mathcal{C}:=\left\{ (C + x )\cap I: x\in \R^d, \quad  I\subset\R^d \text{ is a box} \right\}.\]
\end{definition}
Since $\mathcal{B}(\R^d)$ is itself a ring containing $\mathcal{C}$, observe that $\mathcal{R}\subset \mathcal{B}(\R^d)$. 

\begin{lemma}
    \label{lem:closed}
        For any $R\in \mathcal{R}$ and any box $I\subset \R^d$ one has $R\cap I\in \mathcal{R}$.
    \end{lemma}
    \begin{proof}
        For arbitrary box $I\subset \R^d$ define $\mathcal{F}:=\{R\in \mathcal{R}: R\cap I \in \mathcal{R}\}$. Observe that $\mathcal{C} \subset \mathcal{F}$ and  $\mathcal{F}$ is a ring.
        Therefore, $\mathcal{R}\subset \mathcal{F}$ because $\mathcal{R}$ is the smallest family with that property.
    \end{proof}

    \begin{lemma}
    \label{lem:cover}
    Let $I\subset \R^d$ be a closed box with nonempty interior. Then $I$ cannot be covered by a countable union of elements of $\mathcal{R}$.
    \end{lemma}
        \begin{proof}
            We prove that every set in $\mathcal{R}$ is nowhere dense in $\R^d$ first.

            The idea is to use the following inductive construction of $\mathcal{R}$.
            Let $\mathcal{R}_1:=\mathcal{C}$ and let
            \[\mathcal{R}_{n+1} := \{ A\cup B,\, A\setminus B :  A,B\in \mathcal{R}_n\}.\]
            
            It is known that $\mathcal{R}=\cup_{n\in \N}\mathcal{R}_n$, but we include a short proof for completeness. By induction and definition of a ring, for all $n\in \N$ it follows that $\mathcal{R}_n\subset \mathcal{R}$, so $\tilde{\mathcal{R}}:=\cup_{n\in\N}R_n\subset \mathcal{R}$.
            We now prove that $\tilde{\mathcal{R}}$ is a ring. 
            If $A,B \in \tilde{\mathcal{R}}$, since the family $(\mathcal{R}_n)_{n\in \N}$ is increasing, there exists $n\in\N$ such that $A,B\in \mathcal{R}_n$. Therefore, $A\setminus B$ and $A\cup B$ are in $\mathcal{R}_{n+1}\subset \tilde{\mathcal{R}}$ by the definition. Therefore, $\tilde{\mathcal{R}}$ is a ring that contains $\mathcal{C}$, so $\mathcal{R}\subset \tilde{\mathcal{R}}$.

            We now prove by induction that for every $n\in\N$, every set in $\mathcal{R}_n$ is nowhere dense in $\R^d$. The $n=1$ case is satisfied by the definition of the set $C$. For the induction step we use the fact that union and difference of two nowhere dense sets is again a nowhere dense set. Therefore, using $\mathcal{R}=\cup_{n\in\N}\mathcal{R}_n$, we conclude that every set in $\mathcal{R}$ is nowhere dense in $\R^d$.

            We turn to the proof of the lemma. Suppose that there exists a box $I\subset \R^d$ with nonempty interior and a family $\{R_n\}_{n\in \N}\subset \mathcal{R}$ such that 
            $I\subset \cup_{n\in \N}{R_n}$. 
            Since $I$ has nonempty interior and since each set $R_n$ is nowhere dense in $\R^d$, this would imply that the box $I$ is equal to the countable union of nowhere dense sets, $R_n\cap I$, in $I$, thus contradicting the Baire category theorem. Therefore, the statement of the lemma holds. 
    \end{proof}

We are ready to prove the main theorem. The construction of the measure will closely follow the standard construction of a measure from an outer measure that can be found, for example, in \cite[\S 1]{FollandBook}. However, due to the intentionally nonstandard properties of the ring $\mathcal{R}$, certain subtle modifications are necessary, which we will emphasize during the construction.
\begin{proof}[Proof of Theorem \ref{thm:main}]    
    We define a pre-measure $\mu_0:\mathcal{R}\to [0,\infty]$ by $\mu_0(R):=\lambda(R)$, where $\lambda$ is the Lebesgue measure on $\R^d$. 
    The fact that $\mu_0$ is a pre-measure on $\mathcal{R}$ follows from $\lambda$ being a measure defined on $\mathcal{B}(\R^d)\supset \mathcal{R}$.
    We define the outer measure of an arbitrary set $E\subset \R^d$ with
    \[\mu^*(E):= \inf\left\{\sum_{j=1}^{\infty} \mu_0(R_j): E\subset \bigcup_{j\in\N}R_j, \quad  R_j\in \mathcal{R} \right\}.\]
    Observe that a set $E$ does not necessarily have a countable cover by elements of $\mathcal{R}$. In that case, by a standard convention $\inf \emptyset =+\infty$, one has $\mu^*(E)=+\infty$. 
    
    By a standard argument, as in \cite[Proposition~1.10]{FollandBook}, it follows that $\mu^{*}$ is indeed an outer measure. 
    The necessary inequality $\mu^{*}(\cup_j A_j)\leq \sum_j \mu^{*}(A_j)$ follows trivially in case any of the sets does not have a countable cover by elements of $\mathcal{R}$.

    The set $\mathcal{M}$ of $\mu^{*}$-measurable sets $A\subset\R^d$ satisfying $\forall E\subset \R^d$
    \begin{equation}
    \label{eq:measurability}
        \mu^*(E)= \mu^*(E\cap A)+\mu^*(E\cap A^c).
    \end{equation}
    forms a $\sigma$-algebra by Caratheodory's theorem \cite[Theorem~1.11]{FollandBook}.
    
    Furthermore, by \cite[Proposition~1.13]{FollandBook} (with a trivial modification, namely that the required inequality in part (b) holds trivially when the set $E$ does not have a cover by elements of $\mathcal{R}$) it follows that every element of $\mathcal{R}$ is $\mu^*$-measurable and 
    \begin{equation*}
        \mu^{*}\vert_{\mathcal{R}}=\mu_0 = \lambda\vert_{\mathcal{R}}.
    \end{equation*}
    Specifically, this implies 
    $\mu(C)=\lambda(C)\in (0,\infty)$, 
    so one can take $K:=C$ in the statement of the theorem.

    We now show that $\mathcal{B}(\R^d)\subset \mathcal{M}$. 
    Let $a\in \R$ be arbitrary, and let $A$ be a half-space, i.e. a box of the form $A:=I_1\times \dots\times I_d $, where $I_j=[a,\infty)$ for exactly one $j\in \{1,\dots, d\}$ and $I_k=\R$ for all $k\neq j$. Inequality $\leq$ in \eqref{eq:measurability} follows from subadditivity of outer measure, so we prove the reverse inequality.

    Let $E\subset \R^d$ be arbitrary. If there is no cover of $E$, the left hand side of \eqref{eq:measurability} is equal to $+\infty$, so the inequality holds. Otherwise, for $\eps>0$ arbitrary, choose any $\{R_j\}_{j\in\N}\subset \mathcal{R}$ such that $E\subset\cup_{j\in\N}R_j$ and
    \[\sum_{j=1}^{\infty}\mu_0(R_j)\leq \mu^*(E)+\eps.\]
    Since $A^c$ is again a box, Lemma \ref{lem:closed} implies that both $R_j\cap A$ and $R_j\cap A^c$ are elements of $\mathcal{R}$. Now, the previous inequality, the fact that $\mu_0$ is a pre-measure on $\mathcal{R}$ and the definition of the outer measure imply
    \[\mu^*(E)+\eps \ge \sum_{j=1}^{\infty}\mu_0(R_j) = \sum_{j=1}^{\infty}\left(\mu_0(R_j\cap A) + \mu_0(R_j\cap A^c)\right) \ge \mu^{*}(E\cap A) + \mu^{*}(E\cap A^c).\]
    Since $\eps>0$ was arbitrary, this implies the $\geq$ inequality in \eqref{eq:measurability}. Therefore, $A\in \mathcal{M}$ and consequently $\mathcal{B}(\R^d)\subset \mathcal{M}$, since the half-spaces generate $\mathcal{B}(\R^d)$.
    
    Finally, Lemma \ref{lem:cover} implies that there is no countable cover by elements of $\mathcal{R}$ of any box with nonempty interior, so specifically $\mu([0,1]^d)=+\infty$. Therefore, $\mu \neq \lambda$.
\end{proof}

The following lemma is needed for the proof of Corollary \ref{cor:Hausd}.
\begin{lemma}
\label{lem:cubes}
    Let $\mathcal{Q}=(Q_j)_{j=1}^{n}$ be a family of cubes $Q_j=[0,a_j]^d$ such that $\sum_{j=1}^{n}\lambda(Q_j) \ge 1$. Then, the cube $[0,1/2]^d$ can be covered by translates of cubes from $\mathcal{Q}$.
\end{lemma}
\begin{proof}
    For each $j=1,2,\dots, n$ there exists $k_j\in \Z$ such that $a_j \in [2^{k_j},2^{k_j+1})$. We define another family of cubes $\mathcal{S}:=(S_j)_{j=1}^{n}$, where $S_j=[0,2^{k_j}]^d$. Since $Q_j \subset 2 S_j$, the following estimate holds
    \[\sum_{j=1}^{n}\lambda(S_j) = 2^{-d}\sum_{j=1}^{n}\lambda(2S_j)\ge 2^{-d}\sum_{j=1}^{n}\lambda(Q_j) \ge 2^{-d}.\]
    
    We now apply the following algorithm to the family $\mathcal{S}$. If there exists $k\in \Z$ such that at least $2^d$ cubes in $\mathcal{S}$ have a sidelength $2^k$, arrange them into a cube of sidelength $2^{k+1}$ and replace $2^d$ cubes of sidelength $2^{k}$ with the larger cube. If no such $k$ exists, terminate the algorithm. 
    Since the number of cubes in the family strictly decreases at each step, the algorithm terminates after a finite number of steps. Let the final family of cubes be denoted by $\mathcal{T}=(T_j)_{j=1}^{m}$. The total measure of the cubes remains unchanged, so $\sum_{j=1}^{m}\lambda(T_j)\ge 2^{-d}$,
    and for each $k\in\Z$ there exist at most $2^{d}-1$ cubes in $\mathcal{T}$ with a sidelength equal to $2^k$.
    We claim that the family $\mathcal{T}$ contains a cube with a sidelength at least $2^{-1}$. 
    Indeed, supposing that all cubes in $\mathcal{T}$ have sidelengths strictly less than $2^{-1}$, using the bound for the number of cubes of each given sidelength, this would imply
    \[2^{-d}\leq \sum_{j=1}^{m} \lambda(T_j) < \sum_{k=2}^{\infty} (2^{d}-1) 2^{-kd} = 2^{-d},\]
    leading to a contradiction.
    Finally, since each cube in $\mathcal{T}$ is formed by arranging cubes from $\mathcal{S}$ into a larger cube, and since $S_j\subset Q_j$, this implies the statement of the lemma.
\end{proof}

\begin{proof}[Proof of Corollary \ref{cor:Hausd}]
Suppose that $\nu$ is a Hausdorff measure for which there exists compact $K\in \mathcal{B}(\R^d)$ such that $\nu(K)=\lambda(K)=a\in(0,\infty)$. We claim that $\nu=\lambda$.

We first prove that there exists a cube $Q\subset \R^d$ such that $0<\nu(Q)<\infty$.

Let $\delta > 0$ be arbitrary. From the definition of $\nu_{\delta}^{*}(K)$, we know that there exists a sequence of sets $(E_j)_{j=1}^{\infty}$ such that $K\subset \bigcup_{j\in \N}E_j$, $\operatorname{diam} E_j<\delta$ for all $j\in \N$, and $\sum_{j=1}^{\infty} h(\operatorname{diam} E_j)\leq a + 1$.
Since every set $E\subset \R^d$ is contained in a closed ball of radius $\operatorname{diam} E$ centered at any point of $E$, the following inequality holds: $\lambda(E)\leq C_d(\operatorname{diam} E)^d $. Combining this with the fact that $K\subset \cup_{j\in \N} E_j$, we have:
\[a=\lambda(K)\leq \sum_{j=1}^{\infty}\lambda(E_j) \leq C_d\sum_{j=1}^{\infty}(\operatorname{diam} E_j)^d.\] 
Therefore, we can choose $n\in \N$ large enough so that
\[a/2 < C_d\sum_{j=1}^{n}(\operatorname{diam} E_j)^d.\]
Defining $Q_j:=[0, \operatorname{diam}E_j/\sqrt{d}]^d$, we observe that $\operatorname{diam}Q_j=\operatorname{diam} E_j<\delta$ and 
\[\sum_{j=1}^{n} \lambda(Q_j) = \sum_{j=1}^{n}(\operatorname{diam} E_j/\sqrt{d})^d \ge C_d^{-1}d^{-d/2}a/2 =:\alpha_{a,d}^d.\]
The rescaled version of Lemma \ref{lem:cubes} implies that translates of $(Q_j)_{j=1}^{n}$ can cover a cube $Q:=[0,\alpha_{a,d}/2]^d$. Hence, from the definition of $\nu^{*}(Q)$ and the sets $E_j$, we obtain:
\[\nu_\delta^{*}(Q)\leq \sum_{j=1}^{\infty}h(\operatorname{diam}Q_j) = \sum_{j=1}^{\infty}h(\operatorname{diam}E_j)\leq a+1.\]
Since $\delta>0$ was arbitrary, we conclude that $\nu(Q)=\sup_{\delta >0}\nu_\delta^{*}(Q)\leq a+1<\infty$. 

On the other hand, since $0<\nu(K)\leq \nu(\tilde{Q})$ for some cube $\tilde{Q}\subset \R^d$ large enough to contain $K$, and since $\tilde{Q}$ can be covered by a finite number $N$ of translates of the cube $Q$, it follows that $\nu(Q)\ge \nu(\tilde{Q})/N>0$. Therefore, $0<\nu(Q)<\infty$.

Defining $c:=\lambda(Q)/\nu(Q)$, Lemma \ref{lem:simple_equality} implies $c \nu=\lambda$ on $\mathcal{B}(\R^d)$. However, testing the equality on $K$, one needs to have $c=1$, so $\nu=\lambda$. Therefore, the measure $\mu$ from Theorem \ref{thm:main} cannot be a Hausdorff measure.

Finally, we prove the statement for the range of the measure $\mu$. For $v_1>0$ large enough, denoting $v=(v_1,0,\dots, 0)\in \R^d$, the set $S=\bigcup_{n\in \N}\{C+n\cdot v\}$ is a disjoint union of translates of $C$ because $C$ is compact. Therefore, using the fact that $\mu(C)>0$, it follows that $S$ has infinite measure. Finally, observing that $x\mapsto \mu(S\cap ((-\infty, x]\times \R^{d-1}))$ is a continuous function, the statement follows from the intermediate value property.    
\end{proof}

\section{Remarks on the method of proof}
\begin{remark}
\label{rem}
    It may seem more natural to define $\mathcal{R}$ simply as the ring generated by $\{C+x,\; x\in \R^d\}$. However, with this definition, we were not able to prove that boxes are measurable, so the $\sigma$-algebra of measurable sets with pre-measure defined on such ring might not contain $\mathcal{B}(\R^d)$ and is probably strictly smaller. On the other hand, if one tries to define the pre-measure on a ring larger than $\mathcal{R}$, one needs to be very careful not to define a measure that is equal to the Lebesgue measure.
\end{remark}
\begin{remark}
    We emphasize that it is necessary to use the full power of Caratheodory's theorem \cite[Theorem 1.11]{FollandBook} which guarantees the extension of the pre-measure defined on $\mathcal{R}$ to the $\sigma$-algebra of all $\mu^{*}$-measurable sets, rather than just the $\sigma$-algebra $\sigma(\mathcal{R})$. This is because we cannot prove that the latter contains $\mathcal{B}(\R^d)$. However, unlike the approach in \cite{FollandBook}, we decided to work with the pre-measure $\mu_0$ defined on the ring $\mathcal{R}$ rather than an algebra containing $\mathcal{C}$. This choice simplifies the statement of Lemma \ref{lem:cover}. Specifically, if we had used the algebra instead, we would have needed to consider the case in which the cover of the box $I$ contains a set of the form $(C+x)^c$, which is not a nowhere dense set.
\end{remark}


\section*{Acknowledgements}
Author is grateful to Patrick Pavić and Krešimir Nežmah for raising the main question of the note during the course in Measure theory at University of Zagreb. Author is also grateful to Ren\'{e} Schilling for useful comments that improved the exposition.

This work was supported in part by the Croatian Science Foundation under the project
HRZZ-IP-2022-10-5116 (FANAP).


\bibliography{bibliography}{}
\bibliographystyle{plain}

\end{document}